\documentclass[a4paper]{amsart}

\numberwithin{equation}{section}
\newtheorem{theorem}{Theorem}[section]
\newtheorem{corollary}{Corollary}[section]
\newtheorem{definition}{Definition}[section]

\theoremstyle{remark}

\title[On certain classes of harmonic functions...] {On certain classes of harmonic functions defined by the fractional derivatives }
\author[M. Eshagi Gordji, S.  Shams and A. Ebadian] {M. Eshaghi Gordji, S. Shams and A. Ebadian}
\address {Department of Mathematics, Faculty of Science, Semnan University, Semnan, Iran}
\email{madjid.eshaghi@gmail.com} 
\address {Department of Mathematics, Faculty of Science, Urmia University, Urmia,Iran}
\email{sa40shams@yahoo.com} \email{a.ebadian@mail.urmia.ac.ir}



\begin{document}
\begin{abstract} In this paper we have introduced two new classes
$\mathcal{H}\mathcal{M}(\beta, \lambda, k, \nu)$ and
$\overline{\mathcal{H}\mathcal{M}} (\beta, \lambda, k, \nu)$  of
complex valued harmonic multivalent functions of the form $f = h +
\overline g$, satisfying  the condition
\[ Re \left\{ (1 - \lambda) \frac{\Omega^vf}{z} + \lambda(1-k) \frac{(\Omega^vf)'}{z'} + \lambda k \frac{(\Omega^vf)''}{z''}
\right\} > \beta ,~ (z\in \mathcal{D})\] where $h$ and $g$ are
analytic in the unit disk $\mathcal{D} = \{ z : |z| < 1\}.$ A
sufficient coefficient condition for this function in the class
$\mathcal{H}\mathcal{M}(\beta, \lambda, k, \nu)$ and a necessary and
sufficient coefficient condition for the function $f$ in the class
$\overline{\mathcal{H}\mathcal{M}}(\beta, \lambda, k, \nu)$ are
determined. We investigate inclusion relations, distortion theorem,
extreme points, convex combination and other interesting properties
for these families of harmonic functions.
\end{abstract}

\maketitle
\section{Introduction}
Let $u,v$ be real harmonic function in a simply connected domain
$\Omega$ , then the continuous function $f=u+iv$ defined in $\Omega$
is said to be harmonic in $\Omega$. If $f=u+iv$ is harmonic in
$\Omega$ then there exist analytic functions $G,H$ such that $u=Re~
G$ and $v=Im~ H$ , therefor $f=u+iv=h+\overline g$ where
$h=\frac{G+H}{2},~ \overline g=\frac{\overline G-\overline H}{2}$
and we call $h$ and $g$ analytic part and co-analytic part of $f$
respectively. The jacobian of $f$ is given by
$J_f|z|=|h'(z)|^2-|g'(z)|^2$ , also we show by $w(z)$ the dilatation
function for $f$ and define $w(z)=\frac {g'(z)}{h'(z)}.$ Lewy [6],
Clunie and Small [3] have showed that the mapping $z\longrightarrow
f(z)$ is sense preserving and injective in $\Omega$ if and only if
$J_f|z|>0$ in $\Omega$. The function $f=h+\overline g$ is said to be
univalent in $\Omega$ if the mapping $z\longrightarrow f(z)$ is
sense preserving and injective in $\Omega$. Denote by $\mathcal{H}$
the class of all harmonic functions $f=h+\overline g$ that are
univalent and sense preserving in the open unit disk $\mathcal{D}$
where
\begin{equation}
h(z)=z+\sum_{n=2}^\infty a_nz^n,~ g(z)=\sum_{n=1}^\infty b_n z^n~~
|b_1|<1.
\end{equation}
With normalization conditions $f(0)=0,~ f_z(0)=1$ where $f_z(0)$
denotes the partial derivative of $f(z)$ at $z=0.$ In case $g=0$
this class reduces to the class of $\mathcal{S}$ consisting of all
analytic univalent functions.

\begin{definition} \label{th2.2}( See [7] and [9]) Let the function
$f(z)$ be analytic in a simply-connected region of the $z$-plane
containing the origin. The fractional derivative of $f$ of order
$\nu$ is defined by
\[D_z^\nu f(z)=\frac{1}{\Gamma(1-\nu)}\frac{d}{dz}\int_0^1\frac{f(\zeta)}{(z-\zeta)^\nu}d\zeta, ~~0\leq\nu<1\]
where the multiplicity of $(z-\zeta)^\nu$ is removed by requiring
$\log (z-\zeta)$ to be real when $z-\zeta>0 .$
\end{definition}
Making use of fractional derivative and its known extensions
involving fractional derivatives and fractional integrals, Owa and
Srivastava [8] introduced the operator
$\Omega_z^\nu:\mathcal{A}_0\longrightarrow \mathcal{A}_0$ defined by
\[ \Omega_z^\nu f(z):=\Gamma(2-\nu)z^\nu D_z^\nu f(z) ~~ \nu\neq 2,3,4,...\]
where $\mathcal{A}_0$ denote the class of functions which are
analytic in the unit disk $\mathcal{D}$, satisfying normalization
conditions $f(0)=f'(0)-1=0.$

It is easy to see that
\[ \Omega_z^\nu f(z)=z+\sum_{n=2}^\infty \frac{\Gamma(2-\nu)\Gamma(n+1)}{\Gamma(n+1-\nu)}a_nz^n.~~ f\in \mathcal{A}_0\]

\begin{definition} \label{th2.2}  Suppose that $f=h+\overline g$ where
$h$ and $g$ are in (1.1), define $\Omega_z^\nu f(z)=\Omega_z^\nu
h(z)+\overline {\Omega_z^\nu g(z)}.$
\end{definition}
Then we obtain
\[\Omega_z^\nu f(z)=z+\sum_{n=2}^\infty \frac{\Gamma(2-\nu)\Gamma(n+1)}{\Gamma(n+1-\nu)}a_nz^n+ \sum_{n=1}^\infty \frac{\Gamma(2-\nu)\Gamma(n+1)}
{\Gamma(n+1-\nu)}b_n{\overline z}^n.\]

By making use of Definition 1.2, we introduce a new class of
harmonic univalent functions in the unit disk $\mathcal{D}$ as in
definition 1.3.

\begin{definition} \label{th2.2} Let $\mathcal{H}\mathcal{M}(\beta, \lambda, k, \nu)~ (0\leq
k\leq 1,~0<\beta\leq 1,~0\leq\lambda,~0\leq\nu <1)$ be the class of
functions $f\in \mathcal{H}$ satisfying the following inequality:
\[ Re~ \left\{ (1 - \lambda) \frac{\Omega^vf}{z} + \lambda(1-k) \frac{(\Omega^vf)'}{z'} + \lambda k
\frac{(\Omega^vf)''}{z''} \right\} > \beta .~ (z=re^{i\theta})\]
where
\[z'=\frac{\partial}{\partial\theta}\left(re^{i\theta}\right),~~ z''=\frac{\partial}{\partial\theta}( z')
,\] and
\[(\Omega^\nu f(z))'=
\frac{\partial}{\partial\theta}\left(\Omega^\nu
f(z)\right)=iz(\Omega^\nu h(z))'-i\overline {z(\Omega^\nu g(z))'},\]
\[(\Omega^\nu f(z))''=\frac{\partial}{\partial\theta}(\Omega^\nu
f(z))'=-z(\Omega^\nu h(z))'-z^2(\Omega^\nu h(z))''-\overline
{z(\Omega^\nu g(z))'}-\overline {z^2(\Omega^\nu g(z))''},\] also we
denote by $\overline{\mathcal{H}\mathcal{M}}(\beta, \lambda, k,
\nu)$ the subclass of $\mathcal{H}\mathcal{M}(\beta, \lambda, k,
\nu)$ consisting of functions $f=h+\overline g$ such that
\begin{equation}
h(z)=z-\sum_{n=2}^\infty |a_n|z^n,~~ g(z)=\sum_{n=1}^\infty
|b_n|z^n,~~|b_1|<1.
\end{equation}
\end{definition}
In [9] H. M. Srivastava and S. Owa investigated this class with
$D^\nu f(z)$ instead of $\Omega^\nu f(z)$ where $D^\nu f(z)$ is the
Ruscheweyh derivative of $f$, for $p$-valent harmonic functions.
This class in special cases involve the works studied by the
previous authors such as Bhoosnurmath and Swamay [2], Ahuja and
Jahangiri [1,5].

In this paper the coefficient inequalities for the classes
$\mathcal{H}\mathcal{M}(\beta, \lambda, k, \nu)$ and $\overline
{\mathcal{H}\mathcal{M}}(\beta, \lambda, k, \nu)$ are obtained also
some other interesting properties of these classes are investigated.

\section{Coefficient Bounds}

In the first theorem  we give the sufficient condition for
$f\in\mathcal{H}$ to be in the class $\mathcal{H}\mathcal{M}(\beta,
\lambda, k, \nu).$

\begin{theorem} \label{th2.2} Let $f\in\mathcal{H},$ and
\[\sum_{n=2}^\infty\phi(n,k,\lambda,\nu)|a_n|+\sum_{n=1}^\infty|\psi(n,k,\lambda,\nu)||b_n |<1-\beta,\]where
\begin{equation}
\phi(n,k,\lambda,\nu):=\frac{[1+\lambda(n-1)(1+nk)]\Gamma(n+1)\Gamma(2-\nu)}{\Gamma(n+1-\nu)},
\end{equation}
and
\begin{equation}
\psi(n,k,\lambda,\nu):=\frac{[1-\lambda(n+1)(1-nk)]\Gamma(n+1)\Gamma(2-\nu)}{\Gamma(n+1-\nu)},
\end{equation}
then $f\in \mathcal{H}\mathcal{M}(\beta, \lambda, k, \nu).$ The
result is sharp for the function $f(z)$ given by
\begin{eqnarray}
f(z)&&=z+\sum_{n=2}^\infty\frac{\gamma^n\Gamma(n+1-\nu)z^n}{[1+\lambda(n-1)(1+nk)]\Gamma(n+1)\Gamma(2-\nu)}\nonumber\\
&&+\sum_{n=1}^\infty\frac{\delta^n\Gamma(n+1-\nu)}{|1-\lambda(n+1)(1-nk)|\Gamma(n+1)\Gamma(n-\nu)}\overline
z^n\nonumber
\end{eqnarray}
where $\sum_{n=2}^\infty |\gamma_n|+\sum_{n=2}^\infty
|\delta_n|=1-\beta.$
\end{theorem}
\begin{proof} Suppose
\[E(z)=(1-\lambda)\frac{\Omega^\nu f(z)}{z}+\lambda(1-k)\frac{(\Omega f(z))'}{z'}+\lambda k\frac{(\Omega f(z))''}{z''}.\]
It suffices to show that $|1-\beta+E(z)|\geq |1+\beta-E(z)|.$ A
simple calculation by substituting for $h$ and $g$ in $E(z)$ shows
\begin{eqnarray}
E(z)&&=1+\sum_{n=2}^\infty\frac{[1+\lambda(n-1)(1+nk)]\Gamma(n+1)\Gamma(2-\nu)}{\Gamma(n+1-\nu)}a_nz^{n-1}\nonumber\\&&+
\sum_{n=1}^\infty\frac{[1-\lambda(n+1)(1-nk)]\Gamma(n+1)\Gamma(2-\nu)}{\Gamma(n+1-\nu)}b_n\frac{\overline
z^n}{z},\nonumber
\end{eqnarray}
Considering (2.1) and (2.2) we have
\[
\phi(n,k,\lambda,\nu)=n(n-1)[1+\lambda(n-1)(1+nk)]B(n-1,2-\nu),
\] and
\[
\psi(n,k,\lambda,\nu)=n(n-1)[1-\lambda(n+1)(1-nk)]B(n-1,2-\nu),
\]
where $B(\alpha,\beta)=\int_0^1
t^{\alpha-1}(1-t)^{\beta-1}dt=\frac{\Gamma(\alpha)\Gamma(\beta)}{\Gamma(\alpha+\beta)}$
is the familiar Beta function. Then we obtain
\[E(z)=1+\sum_{n=2}^\infty\phi(n,k,\lambda,\nu)a_nz^{n-1}+\sum_{n=1}^\infty\psi(k,n,\lambda,\nu)b_n\frac{\overline
z^n}{z}.\] Now we have
\begin{eqnarray}
&&|1-\beta+E(z)|-|1+\beta-E(z)|\nonumber\\
&&=|2-\beta+\sum_{n=2}^\infty\phi(n,k,\lambda,\nu)a_nz^{n-1}+\sum_{n=1}^\infty\psi(n,k,\lambda,\nu)b_n\frac{\overline
z^n}{z}|\nonumber\\&&-|\beta-\sum_{n=2}^\infty\phi(n,k,\lambda,\nu)a_nz^{n-1}-\sum_{n=1}^\infty\psi(n,k,\lambda,\nu)b_n\frac{\overline
z^n}{z}|\nonumber\\
&&\geq
2-\beta+\sum_{n=2}^\infty\phi(n,k,\lambda,\nu)|a_n||z|^{n-1}+\sum_{n=1}^\infty|\psi(n,k,\lambda,\nu)||b_n||\frac{\overline
z^n}{z}|\nonumber\\
&&-\beta-\sum_{n=2}^\infty\phi(n,k,\lambda,\nu)|a_n||z|^{n-1}-\sum_{n=1}^\infty|\psi(n,k,\lambda,\nu)||b_n||\frac{\overline
z^n}{z}|\nonumber\\
&&=2-2\beta-2\sum_{n=2}^\infty\phi(n,k,\lambda,\nu)|a_n||z|^{n-1}-2\sum_{n=1}^\infty|\psi(n,k,\lambda,\nu)||b_n||\frac{\overline
z^n}{z}|\nonumber\\
&&>2-2\beta-2\sum_{n=2}^\infty\phi(n,k,\lambda,\nu)|a_n||z|^{n-1}-2\sum_{n=1}^\infty|\psi(n,k,\lambda,\nu)||b_n||\frac{\overline
z^n}{z}|\nonumber\\
&&\geq 0,\nonumber
\end{eqnarray}
and the proof is complete.
\end{proof}
In our next theorem we obtain the necessary and sufficient
coefficients condition for the $f\in\mathcal{H}$ to be in
$\overline{\mathcal{H}\mathcal{M}}(\beta,\lambda,k,\nu).$
\begin{theorem} \label{th2.2} Let $f\in\mathcal{H}$ then
$f\in\overline{\mathcal{H}\mathcal{M}}(\beta,\lambda,k,\nu)$ if and
only if
\begin{equation}
\sum_{n=2}^\infty\phi(n,k,\lambda,\nu)|a_n|+\sum_{n=1}^\infty
|\psi(n,k,\lambda,\nu)||b_n|<1-\beta.
\end{equation}
\end{theorem}
\begin{proof} Since
$\overline{\mathcal{H}\mathcal{M}}(\beta,\lambda,k,\nu)\subset\mathcal{H}\mathcal{M}(\beta,\lambda,k,\nu)$
then the "if" part of theorem follows from Theorem 2.1, for "only
if" part we show that if the condition (2.3) dose not hold then
$f\ne\overline{\mathcal{H}\mathcal{M}}(\beta,\lambda,k,\nu).$ Let
$f\in\overline{\mathcal{H}\mathcal{M}}(\beta,\lambda,k,\nu)$ then we
have
\begin{eqnarray}
0&&\leq Re~ \left\{(1-\lambda)\frac{\Omega^\nu
f(z)}{z}+\lambda(1-k)\frac{(\Omega^\nu f(z))'}{z'}+\lambda
k\frac{(\Omega^\nu f(z))''}{z''}-\beta\right\}\nonumber\\
&&=Re~
\left\{1-\beta-\sum_{n=2}^\infty\phi(n,k,\lambda,\nu)a_nZ^{n-1}-\sum_{n=1}^\infty\psi(n,k,\lambda,\nu)b_n\frac{\overline
z^n}{z}\right\}.\nonumber
\end{eqnarray}
This inequality holds for all values of $z$ for which $|z|=r<1$ so
we can choose the values of $z$ on positive real axis such that
$0\leq z=r<1$ therefore we get the followin inequality
\[
0\leq 1-\beta-\sum_{n=2}^\infty
\phi(n,k,\lambda,\nu)|a_n|r^{n-1}-\sum_{n=1}^\infty
|\psi(n,k,\lambda,\nu)||b_n|r^{n-1}.\] Now by letting
$r\longrightarrow 1^-$ we have
\begin{equation}
0\leq
1-\beta-\sum_{n=2}^\infty\phi(n,k,\lambda,\nu)|a_n|-\sum_{n=1}^\infty
|\psi(n,k,\lambda,\nu)||b_n|.
\end{equation}
If the condition (2.3) dose not hold then the right hand of (2.4) is
negative for $r$ sufficiently close to $1.$ Thus there exists a
$z_0=r_0\in (0,1)$ for which the right hand of (2.4) is negative.
This contradicts the required condition for
$f\in\overline{\mathcal{H}{M}}(\beta,\lambda,k,\nu)$ and so the
proof is complete.
\end{proof}
Putting $\lambda=0$ in Theorem 2.2 we get:
\begin{corollary} \label{th2.2} $f\in\overline{\mathcal{H}\mathcal{M}}(\beta,0,k,\nu)=\left\{f:~ Re~ \left(\frac{\Omega^\nu
f(z)}{z}\right)>\beta\right\}$ if and only if
\[
\sum_{n=1}^\infty n(n-1)B(n-1,2-\nu)|a_n|+\sum_{n=1}^\infty
n(n-1)B(n-1,2-\nu)|b_n|< 1-\beta.
\]
\end{corollary}
Putting $\lambda=1$ in Theorem 2.2 we have:
\begin{corollary}  \label{th2.2} $f\in\overline{\mathcal{H}\mathcal{M}}(\beta,1,k,\nu)=\left\{f:~ Re~ \left((1-k)\frac{(\Omega^\nu f(z))'}{z'}
+k\frac{(\Omega^\nu f(z))''}{z''}\right)>\beta\right\}$ if and only
if \[\sum_{n=2}^\infty
n^2(n-1)(1-k+nk)B(n-1,\nu)|a_n|+\sum_{n=1}^\infty
n^2(n-1)|nk+k-1|B(n-1,2-\nu)|b_n|<1-\beta.\]
\end{corollary}
Putting $k=1$ in Theorem 2.2 we have:
\begin{corollary} \label{th2.2} $f\in\overline{\mathcal{H}\mathcal{M}}(\beta,\lambda,1,\nu)=\left\{f:~ Re~ \left((1-\lambda)\frac{\Omega^\nu f(z)}{z}
+\lambda\frac{(\Omega^\nu f(z))''}{z''}\right)>\beta\right\}$ if and
only if \[\sum_{n=2}^\infty
n(n-1)[1+\lambda(n^2-1)]B(n-1,2-\nu)(|a_n|+|b_n|)<1-\beta.\]
\end{corollary}
Finally putting $k=0$ in Theorem 2.2 we obtain:
\begin{corollary} \label{th2.2} $f\in\overline{\mathcal{H}\mathcal{M}}(\beta,\lambda,0,\nu)=\left\{f:~ Re~ \left((1-\lambda)\frac{\Omega^\nu f(z)}{z}
+\lambda\frac{(\Omega^\nu f(z))'}{z'}\right)>\beta\right\}$ if and
only if \[\sum_{n=2}^\infty
n(n-1)[1+\lambda(n-1)]B(n-1,2-\nu)|a_n|+\sum_{n=1}^\infty
n(n-1)|1-\lambda(n+1)|B(n-1,2-\nu)|b_n|<1-\beta.\]
\end{corollary}
\begin{theorem} \label{th2.2}
$f\in\overline{\mathcal{H}\mathcal{M}}(\beta,\lambda,k,\nu)$ if and
only if
\begin{equation}
f(z)=t_1z+\sum_{n=2}^\infty t_nf_n(z)+\sum_{n=1}^\infty s_ng_n(z) ~~
(z\in\mathcal{D}),
\end{equation}
where $t_i\geq 0,~ s_i\geq 0,~ t_1+\sum_{n=2}^\infty
t_n+\sum_{n=1}^\infty s_n=1$ and
\[f_n(z)=z-\frac{1-\beta}{\phi(n,k,\lambda,\nu)}z^n,\]
\[g_n(z)=z+\frac{1-\beta}{|\psi(n,k,\lambda,\nu)|}\overline z^n.\]
\end{theorem}
\begin{proof} Let $f$ be of the form (2.5) then we have
\begin{eqnarray}
f(z)&&=t_1z+\sum_{n=2}^\infty
t_n\left(z-\frac{1-\beta}{\phi(n,k,\lambda,\nu)}z^n\right)+\sum_{n=1}^\infty
s_n \left(z+\frac{1-\beta}{|\psi(n,k,\lambda,\nu)|}\overline
z^n\right)\nonumber\\
&&=z-\sum_{n=2}^\infty
\frac{1-\beta}{\phi(n,k,\lambda,\nu)}t_nz^n+\sum_{n=1}^\infty
\frac{1-\beta}{|\psi(n,k,\lambda,\nu)|}s_n\overline z^n.\nonumber
\end{eqnarray}
Therefore we have
\begin{eqnarray}
&&\sum_{n=2}^\infty\phi(n,k,\lambda,\nu)
\frac{1-\beta}{\phi(n,k,\lambda,\nu)}t_n+\sum_{n=1}^\infty|\psi(n,k,\lambda,\nu)|
\frac{1-\beta}{|\psi(n,k,\lambda,\nu)|}s_n\nonumber\\
&&=(1-\beta)\left[\sum_{n=2}^\infty t_n+\sum_{n=1}^\infty
s_n\right]=(1-\beta)(1-t_1)\nonumber\\
&&<1-\beta.\nonumber
\end{eqnarray}
This shows that
$f\in\overline{\mathcal{H}\mathcal{M}}(\beta,\lambda,k,\nu).$
Conversely suppose that
$f\in\overline{\mathcal{H}\mathcal{M}}(\beta,\lambda,k,\nu)$ letting
\[t_1=1-\sum_{n=2}^\infty t_n-\sum_{n=1}^\infty s_n,\]
where
\[t_n=\frac{\phi(n,k,\lambda,\nu)}{1-\beta}|a_n|,~
s_n=\frac{|\psi(n,k,\lambda,\nu)|}{1-\beta}|b_n|.\] We obtain
\begin{eqnarray}
f(z)&&=z-\sum_{n=2}^\infty |a_n|z^n+\sum_{n=1}^\infty |b_n|\overline
z^n\nonumber\\
&&=z-\sum_{n=2}^\infty
\frac{1-\beta}{\phi(n,k,\lambda,\nu)}t_nz^n+\sum_{n=1}^\infty
\frac{1-\beta}{|\psi(n,k,\lambda,\nu)|}s_n\overline z^n.\nonumber\\
&&=z-\sum_{n=2}^\infty(z-f_n(z))t_n+\sum_{n=1}^\infty(g_n(z)-z)s_n\nonumber\\
&&=\left(1-\sum_{n=2}^\infty t_n-\sum_{n=1}^\infty
s_n\right)z+\sum_{n=2}^\infty t_nf_n(z)+\sum_{n=1}^\infty
s_ng_n(z)\nonumber\\
&&=t_1z+\sum_{n=2}^\infty t_nf_n(z)+\sum_{n=1}^\infty
s_ng_n(z).\nonumber
\end{eqnarray}
This completes the proof.
\end{proof}
\section{Convolution and Convex combinations}
In the present section we investigate the convolution properties of
the class $\overline{\mathcal{H}\mathcal{M}}(\beta,\lambda,k,\nu).$
The convolution of two harmonic function $f_1$ and $f_2$ given by
\begin{equation}
f_1(z)=z-\sum_{n=2}^\infty |a_n|z^n+\sum_{n=1}^\infty |b_n|\overline
z^n,\\
f_2(z)=z-\sum_{n=2}^\infty |c_n|z^n+\sum_{n=1}^\infty |d_n|\overline
z^n,
\end{equation}
is defined by
\begin{equation}
(f_1*f_2)(z)=z-\sum_{n=2}^\infty |a_nc_n|z^n+\sum_{n=1}^\infty
|b_nd_n|\overline z^n.
\end{equation}
\begin{theorem} \label{th2.2} For $0\leq\beta<\alpha<1 $ let $f_1,~
f_2$ be of the form (3.1) such that for every $n,~ |c_n|<1,~
|d_n|<1.$ If $f_1,~
f_2\in\overline{\mathcal{H}\mathcal{M}}(\alpha,\lambda,k,\nu)$ then
\[f_1*f_2\in\overline{\mathcal{H}\mathcal{M}}(\alpha,\lambda,k,\nu)\subset\mathcal{H}\mathcal{M}(\beta,\lambda,k,\nu).\]
\end{theorem}
\begin{proof} Considering (3.2) we have
\begin{eqnarray}
&&\sum_{n=2}^\infty\phi(n,k,\lambda,\nu)|a_nc_n|+\sum_{n=1}^\infty
|\psi(n,k,\lambda,\nu)||b_nd_n|\nonumber\\
&&<\sum_{n=2}^\infty\phi(n,k,\lambda,\nu)|a_n|+\sum_{n=1}^\infty
|\psi(n,k,\lambda,\nu)||b_n|\nonumber\\
&&<1-\alpha,
\end{eqnarray}
and the proof is complete.
\end{proof}
In the last theorem we examine the convex combination properties of
the elements of
$\overline{\mathcal{H}\mathcal{M}}(\beta,\lambda,k,\nu).$
\begin{theorem} \label{th2.2} The class
$\overline{\mathcal{H}\mathcal{M}}(\beta,\lambda,k,\nu)$ is closed
under convex combination.
\end{theorem}
\begin{proof} Suppose that
\[f_i(z)=z-\sum_{n=2}^\infty |a_{n,i}|z^n+\sum_{n=1}^\infty
|b_{n,i}|\overline z^n,~ i=1,2,...\] then the convex combinations of
$f_i$ may be written as
\[\sum_{i=1}^\infty t_if_i(z)=z-\sum_{n=2}^\infty \left(\sum_{i=1}^\infty
t_i|a_{n,i}|\right)z^n+\sum_{n=1}^\infty \left(\sum_{i=1}^\infty
t_i|b_{n,i}|\right)\overline z^n,\] where $\sum_{i=1}^\infty t_i=1,~
0\leq t_i\leq 1.$ Since
\[\sum_{n=2}^\infty\phi(n,k,\lambda,\nu)|a_{n,i}|+\sum_{n=1}^\infty
|\psi(n,k,\lambda,\nu)||b_{n,i}|<1-\beta,\] so we have
\begin{eqnarray}
&&\sum_{n=2}^\infty\phi(n,k,\lambda,\nu)\left(\sum_{i=1}^\infty
t_i|a_{n,i}|\right)+\sum_{n=1}^\infty
|\psi(n,k,\lambda,\nu)|\left(\sum_{i=1}^\infty t_i|b_{n,i}|\right)\nonumber\\
&&=\sum_{i=1}^\infty
t_i\left\{\sum_{n=2}^\infty\phi(n,k,\lambda,\nu)|a_{n,i}|+\sum_{n=1}^\infty
|\psi(n,k,\lambda,\nu)||b_{n,i}|\right\}\nonumber\\
&&<(1-\beta)\sum_{i=1}^\infty t_i=1-\beta.\nonumber
\end{eqnarray}
This shows that $\sum_{i=1}^\infty
t_if_i(z)\in\overline{\mathcal{H}\mathcal{M}}(\beta,\lambda,k,\nu)$
and the proof is complete.
\end{proof}

\end{document}